\documentclass[12pt,twoside]{amsart}
\usepackage{amsmath,amsthm,amssymb,latexsym}
\usepackage[abbrev]{amsrefs}
\usepackage{color}
\usepackage[normalem]{ulem}
\usepackage{hyperref}
\usepackage{url}

\usepackage{pstricks,pst-plot}
\usepackage{amssymb}
\usepackage{verbatim}
\usepackage{amsmath}
\usepackage{bm}
\usepackage{ulem}
\usepackage{a4wide}
\usepackage[latin1]{inputenc}
\usepackage{times}
\usepackage{amssymb,latexsym}
\usepackage{enumerate}
\usepackage[stable]{footmisc}
\usepackage{color}
\usepackage{enumitem}
\usepackage{cancel}
\usepackage{color}
\usepackage{cancel}
\usepackage{tikz}
\usepackage{float}
\usepackage{enumitem}
\usepackage{cancel}
\usepackage[misc]{ifsym}
\usepackage{tikz, bbding}
\definecolor{Black}{cmyk}{0,0,0,1}
\definecolor{OrangeRed}{cmyk}{0,0.6,1,0} 
\definecolor{DarkBlue}{cmyk}{1,1,0,0.20}
\definecolor{myblue}{rgb}{0.66,0.78,1.00}
\definecolor{Violet}{cmyk}{0.79,0.88,0,0}
\definecolor{Lavender}{cmyk}{0,0.48,0,0}

\def\Xint#1{\mathchoice
{\XXint\displaystyle\textstyle{#1}}%
{\XXint\textstyle\scriptstyle{#1}}%
{\XXint\scriptstyle\scriptscriptstyle{#1}}%
{\XXint\scriptscriptstyle\scriptscriptstyle{#1}}%
\!\int}
\def\XXint#1#2#3{{\setbox0=\hbox{$#1{#2#3}{\int}$ }
\vcenter{\hbox{$#2#3$ }}\kern-.6\wd0}}

\def\dashint{\Xint-}

\newcommand{\R}{\ensuremath{\mathbb{R}}}

\newcommand{\C}{\ensuremath{\mathbb{C}}}

\newcommand{\bea}{\begin{eqnarray*}}
\newcommand{\eea}{\end{eqnarray*}}

\numberwithin{equation}{section}

\newtheorem{theorem}{Theorem}[section]
\newtheorem{proposition}[theorem]{Proposition}

\newtheorem{lemma}[theorem]{Lemma}

\newtheorem{definition}{Definition}
\numberwithin{definition}{section} 
\newtheorem{remark}[theorem]{Remark}

\newtheorem{example}[theorem]{Example}

\title[Equivalence between VMO functions and Zero Lelong number functions]{Equivalence between VMO functions and plurisubharmonic functions with zero Lelong numbers}
\keywords{plurisubharmonic functions, VMO functions, Lelong number}
\date{\today}
\subjclass[2010]{32U25, 31C10, 30H35}

\author{S\'everine Biard}
\address{Univ.  Polytechnique Hauts-de-France,  INSA Hauts-de-France,  Ceramaths, FR CNRS 2956, F-59313 Valenciennes.}
\email{severine.biard@uphf.fr}
\author{Jujie Wu}
\address{School of Mathematics (Zhuhai), Sun Yat-Sen University, Zhuhai, Guangdong 519082, P. R.
 	China}
	\email{wujj86@mail.sysu.edu.cn}
	
\thanks{This study was supported by National Key R\&D Program of China,  No.  2024YFA1015200. The first author is partially supported by the ANR QuaSiDy, grant no ANR-21-CE40-0016. The Second author is also partially supported by the Natural Science Foundation of Guangdong Province, No.  2023A1515030017.}

\begin{document}
\begin{abstract}

We prove that a plurisubharmonic function on a domain in the complex Euclidean space  
is a locally  VMO  (Vanishing Mean Oscillation) function if and only if its Lelong number at each point vanishes. We also give a global version of this result when the boundary of the domain satisfies the \textit{interior sphere condition}. An example emphasizes the importance of this condition. These equivalences contribute to a better understanding of the behavior of singular plurisubharmonic functions. We end the paper by discussing the link between the residual Monge-Amp\`ere mass and VMO functions, by providing examples.
\end{abstract}

\maketitle

\tableofcontents  

\section{Introduction}
  Any plurisubharmonic function $\varphi$ in a domain $\Omega\subset \mathbb{C}^n$ defines a positive closed $(1,1)$-current $dd^c\varphi$. The Lelong number of $\varphi$ at a point $a\in\Omega$, denoted by $\nu_\varphi(a)$ is then the $(2n-2)$ dimensional density of its Riesz measure $\frac{1}{2\pi}\Delta\varphi$ at $a$.   
  Introduced by Lelong in 1957 \cite{Lelong57}, the Lelong number provides a beautiful link between analytical and geometrical objects in complex analysis. 
 This quantity is invariant by local change of coordinates and quantifies the singularities. 
In \cite{Skoda1972}, Skoda states a relation between the Lelong number and the integrability of $e^{-\varphi}$,  proving that, if $ \nu_\varphi(a)< 2$, then the function $e^{-\varphi}$ is locally integrable with respect to the Lebesgue measure in a neighborhood of $a$. Applications to pluripotential theory and complex dynamics are then launched and spark interest of numerous mathematicians (see for example \cites{Zeriahi2001, DK2001, BFW2019, Chen2020, BFW2021, Chen2021}). 
Using the complex Monge-Amp\`ere operator defined by Bedford and Taylor, Demailly \cite{Demailly1982} generalizes the Lelong numbers for positive closed $(p,p)$-currents and then develops techniques and results used in algebraic geometry and number theory.  
 
 In particular, zero Lelong numbers appeared recently in various areas, such as the dynamics of holomorphic foliations \cite{Nguyen2023} and raise numerous questions \cite{Rash2016, Ra2016} such as its implication in a zero residual Monge-Amp\`ere mass \cite{Li2023}, also called the zero mass conjecture. We present in this paper a refined study of the
properties of plurisubharmonic functions with zero Lelong numbers, applying methods of real analysis in the complex setting and connecting real harmonic analysis notions to the pluripotential theory.  

%
 
\indent A locally integrable function in a domain $\Omega$ in $\mathbb{C}^n$ is a Bounded Mean Oscillation function (BMO) if its mean oscillation, which means the local average of its pointwise deviation from its mean value,  is bounded over the balls of $\Omega$.  These functions were introduced by John and Nirenberg \cite{John-Nirenberg1961} to measure the sets of points where the distance between the function and its mean value is bigger than a positive constant, over subcubes. In the early seventies, D. Sarason \cite{Sarason1975} introduced Vanishing Mean Oscillation (VMO) functions for their applications to stochastic processes. 
A VMO function is a BMO function whose mean oscillation converges to zero when the volume of the balls goes to zero. Although such functions are still extremely useful in harmonic analysis, it is only recent they reach complex variables, thanks to \cites{Chen2020,Chen2021}. Indeed, 
a plurisubharmonic function verifies the mean value inequality, which makes it a natural candidate for BMO functions.
Here, we provide a characterization of the plurisubharmonic VMO functions on domains in $\mathbb{C}^n$ in terms of the following quantity, known as the Lelong number.\\




  
\begin{definition}
    Let $\Omega$ be a bounded domain in $\C^n$ and $\varphi$ be a plurisubharmonic function on a neighborhood of $\overline{\Omega}$. Then
\begin{eqnarray*} 
  \nu_{\varphi} (\overline{\Omega}) : = \sup \{  \nu_\varphi (a)\mid  a \in \overline{\Omega} \}.  
\end{eqnarray*}
\end{definition}
 
Considering the points on the boundary of $\Omega$, 
our result requires the following assumption: $\Omega$ is said to satisfy the \textit{interior sphere condition} (see for example \cite{GT}) if for every $a\in \partial \Omega$,  there exists a ball $B \subset \Omega$ such that $a\in \partial B$. 

%
%
 We then prove
 
\begin{theorem}\label{tmthm2}
Let $\Omega$ be a bounded domain in $\C^n$ satisfying the \textit{interior sphere condition} at each point of $\partial \Omega$ and let $\varphi$ be a plurisubharmonic function in a neighborhood of $\overline{\Omega}$.  
Then $\nu_\varphi ( \overline{\Omega}) =0$ if and only if $\varphi$ is \rm{VMO} on $\Omega$.
 \end{theorem}

We give an example (Example \ref{example1}) of a plurisubharmonic VMO function on a H\"older domain in $\mathbb{C}$ with a nonzero Lelong number, emphasizing the importance of the interior sphere condition. But without any additional geometric assumption on $\Omega$, we obtain a local version.
 
\begin{theorem}\label{tmthm1}
Let $\Omega\subset \C^n$ be a domain.   Suppose that $\varphi$ is a plurisubharmonic function on $\Omega$,  then $ \nu_\varphi (a)=0$ for  every  $a\in \Omega$ if and only if $\varphi$ is locally \rm{VMO} on $\Omega$.
\end{theorem}
 
This theorem helps to construct numerous examples of plurisubharmonic VMO functions (see Section 5). Connecting VMO functions to the zero Lelong number lead us to believe that it will help to solve the zero mass conjecture.  
Indeed, $\mathrm{W^{1,2}_{loc}}$ functions are VMO in the complex plane and $\mathrm{W^{1,2}_{loc}}$ plurisubharmonic functions are in the domain of definition of the complex Monge-Amp\`{e}re operator in $\mathbb{C}^2$. However, we give an example of a plurisubharmonic and locally VMO function in a domain in $\mathbb{C}^2$ that is not in $\mathrm{W^{1,2}_{loc}}$ (see Remark \ref{rmk:notW}).

\indent Before proving Theorem \ref{tmthm2} in Section 5, we recall in Section 2 the definitions of the objects we use, like BMO, VMO functions, and the various definitions of the Lelong number in Section 3, as well as some of its properties and connections to the integrability index. We then give preliminary results connecting the BMO norm and the Lelong number through the John-Nirenberg inequality in Section 4.  We end the paper with Section 6, constructing plurisubharmonic functions that are $\mathrm{W^{1,2}_{loc}}$ and giving, with Remark \ref{rmk:notW}, a plurisubharmonic and locally VMO function in $\mathbb{C}^2$ that is not in the domain of definition of the complex Monge-Amp\`{e}re operator.

\section{BMO and VMO functions}
 
In this paper, $\Omega$ denotes a domain in $\C^n$ as an open and connected set in $\mathbb{C}^n$ and $\mathrm{PSH}(\Omega)$ denotes the set of plurisubharmonic functions on $\Omega$, that are not identically $-\infty$. We denote the mean value of any locally integrable ($L^1_{\rm{loc}}$) function $\varphi$ over a ball 
$B(z, r) \subset \Omega$  with center $z\in\Omega$ and radius $r$ by
 
$$\varphi_{B(z, r)}: =  \frac{1}{|B(z, r)|} \int_{B(z, r) }\varphi.$$
Here,  $|B(z, r)|$ is the Lebesgue measure of $B(z, r)$ and $\displaystyle{\int_{B(z, r)}}$ should be understood as the Lebesgue integral.  In general, we denote $\displaystyle{ \dashint_ { E } \varphi:=\frac{1}{\vert E \vert}\int_{E }\varphi}$ for a non-empty set $E\subset \mathbb{C}^n$. \\

The mean oscillation of $\varphi$ at a point $z\in\Omega$ is defined to be 

\begin{align}
\mathrm{MO}_{B(z,r)}(\varphi)  := \   \dashint_{B(z,r)} |\varphi - \varphi_{B(z,r)}|,
\end{align}
for   $  B(z,r)  \subset \Omega$.  
We say that  $\varphi \in L^1_{\rm{loc}}(\Omega)$  has \textit{Bounded Mean Oscillation} (BMO) on $\Omega$  (or equivalently $\varphi$ is a BMO function on $\Omega$, denoted $ \mathrm{BMO }(\Omega)$)  if  

\begin{align}\label{BMO1}
    \| \varphi\|_{\mathrm{BMO}(\Omega)} :=  \sup_{ B  \Subset \Omega} \mathrm{MO}_{B }(\varphi)    =   \sup_{ B  \Subset \Omega}    \dashint_{B } |\varphi - \varphi_{B }| < \infty,
\end{align}
where the supremum is taken over all balls $B\Subset\Omega$. Some authors \cite{Chen2021} specify that such a function is BMO \textit{with respect to balls} (see Remark \ref{RMBMO}). It will be implied in this paper.   If $\varphi=c$ on $\Omega$, $c\in\mathbb{C}$, then $\Vert \varphi\Vert_{\mathrm{BMO}(\Omega)}=0$, so $\Vert .\Vert_{\mathrm{BMO}(\Omega)}$ is a norm on $L^1_{\rm{loc}}/\{\rm{constants}\}$. In the paper, we will identify a $\mathrm{BMO}$ function $\varphi$ to $\varphi + c$, $c\in\mathbb{C}$ and then refer to $\Vert .\Vert_{\mathrm{BMO}(\Omega)}$ as the $\mathrm{BMO}$ norm on $\Omega$.

According to P. Jones \cite{Jones} (see also \cite{BN1996}, Theorem A1.1), there exists a constant $C>0$ that depends on $n$ and the choice of the norm of $\C^n \cong\mathbb{R}^{2n}$ such that

\begin{align}\label{BMO0}
       \sup_{\tiny{\begin{gathered}
 z\in\Omega,\\  
  B(z,r)\Subset \Omega,\\
 r \leq \frac{1}{2}\mathrm{dist}(z, \partial \Omega)
  \end{gathered}}} \mathrm{MO}_{B(z,r)}(\varphi)
  \leq \Vert f\Vert_{\mathrm{BMO}(\Omega)} \leq C\sup_{\tiny{\begin{gathered}
 z\in\Omega,\\  
  B(z,r)\Subset \Omega,\\
 r \leq \frac{1}{2}\mathrm{dist}(z, \partial \Omega)
  \end{gathered}}} \mathrm{MO}_{B(z,r)}(\varphi).  
\end{align}
By a covering argument, we may even take for $r_0>0$, the supremum over the balls $B(z,r)$ with $r\leq \min(r_0, k\mathrm{dist}(z,\partial\Omega))$, for any $0<k<1$, and get the equivalence with the BMO norm (see Corollary A1.1 in \cite{BN1996}). For convenience, we will use the BMO norm defined as follows. 
\begin{align}\label{BMO2}
    \| \varphi\|_{\mathrm{BMO}(\Omega)} :=    \sup_{\tiny{\begin{gathered}
 z\in\Omega,\\  
  B(z,r)\Subset \Omega,\\
 r \leq \frac{1}{2}\mathrm{dist}(z, \partial \Omega)
  \end{gathered}}} \mathrm{MO}_{B(z,r)}(\varphi).
\end{align}

A BMO function $\varphi$ on $\Omega$ has \textit{Vanishing Mean Oscillation} (VMO) on $\Omega$ (or equivalently $\varphi$ is a VMO function on $\Omega$, denoted $ \mathrm{VMO }(\Omega)$) if  $\varphi \in  \mathrm{BMO } (\Omega)$ and 

 \begin{equation}\label{VMO}
  \lim_{\tiny{\begin{gathered}
  r\rightarrow 0^{+},\\  
 r \leq \frac{1}{2}\mathrm{dist} (z, \partial \Omega)
  \end{gathered}}}    \dashint_{B(z,r)  } |\varphi - \varphi_{B(z,r)  }|  =0  \ \ \   \text{uniformly with respect to $z\in\Omega$}. 
  \end{equation}
More precisely,   (\ref{VMO}) means that for every $\varepsilon >0$,  there exists $\delta>0$ such that,  for all $z\in \Omega$,

$$
 \dashint_{B(z,r)  } |\varphi - \varphi_{B(z,r)  }|   < \varepsilon
$$
for all $r \leq \min\{\delta, \frac{1}{2} \mathrm{dist} (z, \partial \Omega) \}$.

 \begin{remark}\label{RMBMO}
 The functions $\mathrm{BMO }$ have originally been defined for a cube in $\mathbb{R}^n$ with respect to subcubes \cite{John-Nirenberg1961}.  
 However, the definition of  $\mathrm{BMO }$  functions can be generalized for any relatively compact open subsets in a domain $\Omega\subset \mathbb{C}^n$ \cite{Chen2021}.  For more properties about $\mathrm{BMO }$ and $\mathrm{VMO }$,  we refer to \cite{BN1961} and \cite{BN1996}. There  one can also find the characterization of  $\mathrm{VMO }$ functions as those which are in the closure of $C^\infty$ functions with compact support in $\Omega$ with respect to the $\mathrm{BMO }$ norm.
 \end{remark}

Let $\mathrm{BMO_{loc}}(\Omega)$ (resp. $\mathrm{VMO_{loc}}(\Omega)$) be the set of functions on $\Omega$ that are BMO (resp. VMO) on every open set  
 $\Omega_0\Subset\Omega$.  
By using potential theoretic methods,  Brudnyi \cites{Bru1999, Bru2002} showed that each plurisubharmonic function on $\Omega$ is in $\mathrm{BMO_{loc}}(\Omega)$ (see also \cite{Chen2020} for another proof).

In \cite{Chen2021}, Chen and Wang introduced a new family of plurisubharmonic functions, looking at the difference between the function $\varphi$ and $\sup_{B(z,r)}  \varphi$ (its upper oscillation) instead of $\varphi_{B(z,r)}$ (the mean oscillation): A function $\varphi \in L^1_{\rm{loc}} (\Omega)$ has \textit{Bounded Upper Oscillation} (BUO) on $\Omega$  if 
\begin{eqnarray*}
 \|\varphi\|_{\mathrm{BUO}(\Omega)} := \sup_{ B  \Subset \Omega} \mathrm{UO}_{B }(\varphi)   < \infty
\end{eqnarray*}
where 

\begin{align}\label{UO}
\mathrm{UO}_{B(z,r)}(\varphi)  :=  \dashint _{B(z,r)} |\varphi - \sup _{B(z,r)} \varphi | = \sup_{B(z,r)} \varphi - \varphi _{B(z,r)}.
\end{align}

Note that a BUO function is BMO. Like BMO functions, we will analogously use the notation $\mathrm{BUO_{loc}}(\Omega)$.\\

\section{The Lelong number and the integrability index}
Let $\Omega\subset \mathbb{C}^n$ be a domain. It is well known that 
$\rm{PSH}(\Omega)\subset L^1_{\rm{loc}} (\Omega)$. 
The Lelong number of $\varphi$, denoted by $\nu_\varphi$, measures the size of the singularity of a plurisubharmonic function at a point $z\in\Omega$ and can be defined after Kiselman \cite{Kiselman1979} as
$$ \nu_\varphi(z):=\lim_{r\to 0^+} \dfrac{\max_{\vert \zeta-z\vert =r} \varphi(\zeta)}{\log r}.$$

The limit always exists since $\log r  \mapsto \max_{\vert \zeta-z\vert\le r}\varphi(\zeta)$ is convex, non-decreasing, and bounded from above in $\R$ [\cite{GZ2007}, p.54, Corollary 2.29].
From this definition, it follows that $\nu_\varphi(a)=0$ if $\varphi(a)>-\infty$.
Let $z\in\Omega$ and $t<\log \text{dist}(z,\partial\Omega)$.
The Lelong number can also be seen as the slope of the functions

 $$\displaystyle{\lambda(\varphi, z,t):= \dashint_{\mathbb{S}_n}\varphi(z+xe^t)d\sigma(x)}$$
 and
 
 $$\displaystyle{\Lambda(\varphi, z,t):=\sup\lbrace{\varphi(x)\mid x\in B(z,e^{t})}\rbrace}$$ as $t$ tends to $-\infty$, i.e.,

\begin{equation}\label{slope}
 \nu_\varphi(z)=\lim_{t\to -\infty}\dfrac{\lambda(\varphi, z,t)}{t}=\lim_{t\to -\infty}\dfrac{\Lambda(\varphi, z,t)}{t}.
\end{equation}
Here, $\mathbb{S}_n$ is the unit sphere in $\mathbb{C}^n$, $B(z,e^t)$ is the ball of center $z$ and radius $e^t$ and $d\sigma$ is the Lebesgue measure on smooth real hypersurfaces in $\mathbb{C}^n$.

By choosing $r=e^t$ in \eqref{slope}, we get.

\begin{lemma}[\cites{Kiselman1979, Kiselman1994}]\label{lelongnumber}
 Let $\varphi\in\rm{PSH}(\Omega)$ and $a \in \Omega$. Then, 
 
\begin{align*}
 \nu_\varphi(a)= \lim \limits _{r \rightarrow 0^+}   \dfrac{  \dashint _ {  \partial B(a,r) }  \varphi \, d\sigma}   { \log r } =  \lim \limits _{r \rightarrow 0^+}   \dfrac{\dashint _ {  B(a,r) }  \varphi }   { \log r } =  \lim \limits _{r \rightarrow 0^+} \dfrac{ \sup_{B(a,r)}  \varphi } { \log r}.
\end{align*}
\end{lemma}

We extend the definition of the Lelong number to $\overline{\Omega}$ as follows.

\begin{definition}
    
Let $\Omega$ be a bounded domain in $\C^n$ and $\varphi$ be a plurisubharmonic function on a neighborhood of $\overline{\Omega}$. Then
\begin{eqnarray*} \label{df:uniformll}
  \nu_\varphi (\overline{\Omega}) : = \sup \{  \nu_\varphi(a)\mid  a \in \overline{\Omega} \}.  
\end{eqnarray*}
\end{definition}
By a result from \cite{Siu1974}, $\nu_\varphi (\overline{\Omega})$ is finite.

The integrability index, denoted by $\iota_\varphi$, also measures the singularity of a plurisubharmonic function $\varphi$ at a point $a$ but in terms of the exponential integrability. 
\begin{eqnarray*}\label{integrabelindex}
\iota_\varphi (a) : = \inf \{ r >0 \mid e^{-\frac{2\varphi} {r}} \ \text{is} \ \  L^1  \ \  \text{on a neighborhood of}  \ \  a  \}.
\end{eqnarray*}
If we assume that $\varphi$ is not identically equal to $-\infty$ near $a$, then $\iota_\varphi (a) \in [0, + \infty)$. 
According to Demailly and Koll$\rm{\acute{a}}$r \cite{DK2001},  the following sharp estimate is obtained from Skoda's results \cite{Skoda1972}, relating the Lelong number to the integrability index:
\begin{eqnarray}\label{iota}
  \frac 1n \nu_\varphi(a)  \leq \iota_\varphi (a)  \leq \nu_\varphi(a) .
\end{eqnarray}

\section{BMO functions and Lelong number}

In this section, we connect the Lelong number with the BMO norm. The results are utilized in the proof of Theorem \ref{tmthm2} given in the next section.

We start with a characterization of  VMO  functions in terms of the BMO norm.

\begin{lemma}  \label{LM:BMOEPSILON}
Let $\Omega$ be a domain in $\C^n$.  If $\varphi \in  \mathrm{VMO }(\Omega)$,   then we have 
\begin{eqnarray}
\lim_{r \rightarrow 0} \| \varphi \| _{\mathrm{BMO}(B(a, r))} = 0 \ \ \ \ \text{uniformly with respect to $a\in\Omega$}.
\end{eqnarray} 
\end{lemma}

\begin{proof}
Let $B(a,r) \subset \Omega$.  From \eqref{BMO2}, we have 
\begin{eqnarray*}
\| \varphi \| _{ \mathrm{BMO} ( B(a,r))} =\sup _ {{\tiny{\begin{gathered} B(\zeta, t) \Subset  B(a,r),\\ t \leq \frac{1}{2} \mathrm{dist} (\zeta,   \partial B(a,r) )\end{gathered}}} }   \dashint_{B(\zeta,t)  } |\varphi - \varphi_{B(\zeta,  t)  }|.
\end{eqnarray*} 
Since $\varphi \in  \mathrm{VMO }(\Omega)$, then for every $\varepsilon >0$,  there exists $\delta>0$ such that,  for all $a\in \Omega$,
$$
  \dashint_{B(a,r)  } |\varphi - \varphi_{B(a,r)  }|   < \varepsilon
$$
for all $r \leq \min\{\delta,  \frac{1}{2} \mathrm{dist} (a, \partial \Omega) \}$. So, if $r < \delta$,  for every   $B(\zeta, t) \Subset  B(a,r)$ with  $t \leq   \frac{1}{2} \mathrm{dist} (\zeta, \partial B(a,r))$,   we have 
$$
t \leq \frac{1}{2}\mathrm{dist} (\zeta,  \partial B(a,r))   \leq \frac{1}{2} \delta,
$$
and $t \leq \frac{1}{2}   \mathrm{dist} (\zeta, \partial \Omega)$.  Thus,  
$$
  \dashint_{B(\zeta,t)  } |\varphi - \varphi_{B(\zeta,  t)  }|< \varepsilon.
$$
Hence,  for every $\varepsilon >0$,  there exists $\delta>0$ which is independent of $a$,  so that, if $r<\delta$, then
\begin{eqnarray*}
\| \varphi \| _{ \mathrm{BMO} ( B(a,r))} \leq  \varepsilon.
\end{eqnarray*} 

\end{proof}

The following inequality, initially stated over subcubes of $\mathbb{R}^n$ and commonly called John-Nirenberg inequality,  was proved by John and Nirenberg \cite{John-Nirenberg1961} and provides a strong relationship between  BMO  norm and the exponential integrability.

  \begin{lemma}[John-Nirenberg inequality]\label{JNE1}
 
 Let $\Omega$ be a domain in $\mathbb{C}^n$ and 
 $\varphi\in L^1_{\rm{loc}}(\Omega)\cap \mathrm{BMO}(\Omega)$. Then there exist $c_n>0$ and $C_n>0$, depending only on $n$ such that 
 \begin{equation}\label{JNE}
   \sup _{B\subset  \Omega}\dashint_B e^{c_n\frac{ 2\vert \varphi-\varphi_B \vert}{\Vert\varphi\Vert_{\mathrm{BMO}(\Omega)}}}\leq C_n,
 \end{equation}
 where the supremum is taken over all the balls $B\subset \Omega$.
 \end{lemma}

Since $\rm{PSH}(\Omega)\subset \rm{BMO_{loc}}(\Omega)$, the John-Nirenberg inequality holds on $\rm{PSH}(\Omega)$.  
Lemma \ref{JNE1}  implies the existence of a positive constant $c_n$ depending only on the dimension $n$ such that for every domain $\Omega$, we have
$$
    e^{-c_n\frac{2\varphi}{ \Vert \varphi\Vert_{\mathrm{BMO}(\Omega )}}}\in     L^1_{\text{loc}}  (\Omega)  .
 $$
By definition, we deduce the following inequality on the integrability index at $a\in\Omega $,
\begin{eqnarray}\label{iotaBMO}
\iota_\varphi (a)  \leq  \frac{ \| \varphi \| _{ \mathrm{BMO}(\Omega) }}{ c_n} ,
\end{eqnarray}
 where $c_n>0$ only depends on the dimension $n$. 

Combining \eqref{iotaBMO} obtained by Lemma \ref{JNE1} on $B(a,r)\subset \Omega$ with Lemma \ref{LM:BMOEPSILON} and \eqref{iota}, we get. 

\begin{proposition}\label{co:vmoio}
Let $\Omega\subset \C^n$ be a bounded domain and $\varphi\in \rm{PSH}(\Omega)$.  If $\varphi \in  \mathrm{VMO }(\Omega)$, then $ \nu_\varphi (a)=0$ for any $a\in \Omega$.
\end{proposition}
 
\begin{proof}
Since $\varphi\in{\mathrm{VMO}}(\Omega)$, by Lemma \ref{LM:BMOEPSILON}, we deduce that $\lim_{r\to 0^+}\Vert \varphi\Vert_{\mathrm{BMO}(B(a,r))}=0$ for any ball $B(a,r)\subset \Omega$.
By Lemma \ref{JNE1}, we obtain \eqref{iotaBMO} on $B(a,r)$: there exists $c_n>0$ such that for all $a\in\Omega$,
$$\iota_\varphi (a)  \leq  \frac{ \| \varphi \| _{ \mathrm{BMO}(B(a,r)) }}{ c_n}.
$$
When $r$ tends to $0^+$, the right-hand side of the previous inequality tends to $0$, then $\iota_\varphi (a)=0$ and by \eqref{iota}, it implies that
$\nu_\varphi(a)=0$.
\end{proof}

\section{Proof of Theorem \ref{tmthm2} and Theorem  \ref{tmthm1}}

The proof of Theorem \ref{tmthm2} relies on ideas from \cite{Chen2021}, based on Harnack inequality (see for example p.37 in \cite{Dem97}) and convex analysis that we recall below.  

\begin{lemma}[Harnack inequality]\label{supu}
Let $\varphi $ be a subharmonic function on $\Omega \subset \C^n$ and $\bar{B}(0,r) \subset \Omega$. 
If $\varphi \leq 0$ on  $\bar{B}(0,r) $,  then for all $x\in B(0,r)$
\begin{eqnarray}
\varphi (x)  
 \leq   \frac{r ^{2n -2 } (r- |x|)}{ (r+|x|)^{2n-1}}   \dashint_ {  \partial B(0,r) }  \varphi \, d\sigma.
\end{eqnarray}
 
\end{lemma}

We also need the following lemma, that may be seen as Jensen's inequality.


\begin{lemma}[Lemma 3.0.2 in \cite{Chen2021}]\label{barycenter}
 Let $d\mu$ be a probability measure on a Borel measurable subset $S$ in $\mathbb{R}^n$ with
barycenter $\hat{t}$. Let $g$ be a convex function on $\mathbb{R}^n$. Then
\begin{eqnarray}
 \int _S g d \mu \geq g( \hat{t}).
\end{eqnarray}
\end{lemma}

\begin{lemma}
Let $\Omega$ be a domain in $\C^n$ and $B(a,r) \subset  \Omega$.  If $\varphi\in L^1_{\rm{loc}}(\Omega)$, then
   \begin{align}\label{eq:mo}
 \mathrm{MO}_{B(a,r)}(\varphi)  \leq &  2  \left(\dfrac{3^{2n-1}}{ 2^{2n-2}}\right) \left(\sup _{B(a,r)} \varphi - \sup_{B(a,\frac{1}{2}r)}\varphi\right)  +2\left ( \varphi_ { \partial B(a, r)}  -   \varphi_ { \partial B\left(a,e^{-\frac{1} { 2n } }r \right)} \right).
\end{align}
 
\end{lemma}

  \begin{proof}
We may assume that $a=0$.  We remark that for any $B(0,r)\subset \Omega$,
\begin{eqnarray} \label{eq:uomo}
\mathrm{MO}_{B(0,r)}(\varphi) & = &  \dashint_{B(0,r) }  |\varphi - \varphi_{B(0,r) }|  \\ \nonumber
&= &  \dashint _{B(0,r)}  |\varphi - \sup_{ B(0,r) }  \varphi + \sup_{B(0,r) }  \varphi -  \varphi_{B(0,r) }| \\ \nonumber
& \leq  &  \dashint_{B(0,r)}  |\varphi - \sup_{B(0,r)}  \varphi |  \\  \nonumber
&& +  \dashint_{B(0,r) }  |\sup_{B(0,r) }  \varphi -  \varphi_{B(0,r) }| \\ \nonumber
&=   &  2 \left ( \sup_{B(0,r) }  \varphi -   \dashint_{B(0,r) }   \varphi \right )\\ \nonumber 
& =& 2  ( \sup_{B(0,r)} \varphi - \varphi_{B(0,r) }) \\ \nonumber 
&= & 2 \mathrm{UO}_{B(0,r)}(\varphi) .  
\end{eqnarray}  
We decompose now $\mathrm{UO}_{B(0,r)}(\varphi)$ in terms that we can estimate by the Lelong number of $\varphi$ at $0$.
\begin{eqnarray} \label{UO1}
 \mathrm{UO}_{B(0,r)}(\varphi)  &=& \sup_{B(0,r)} \varphi -  \varphi_{ \partial B(0,r)} +  \varphi_{ \partial B(0,r)}  -  \varphi_{B(0,r)}\\   \nonumber
&: = & I_1 + I_2.
\end{eqnarray}
By using Lemma \ref{supu} for $\Psi:= \varphi - \sup _{B(0,r)} \varphi $, we get that 
\begin{eqnarray*} 
\sup  _{B(0,\frac 1 2 r)} \Psi & \leq &\frac 1 2\cdot \left (\frac 2 3\right)^{2n-1}  \Psi_{\partial B(0,r)} .
\end{eqnarray*}
That is,
\begin{align*}
\sup_{B(0,\frac{1}{2}r)} \Psi= \sup_{B(0,\frac{1}{2}r)}\varphi - \sup _{B(0,r)} \varphi &\leq\frac 1 2\cdot \left (\frac 2 3\right)^{2n-1}  ( \varphi _{\partial B(0,r)}- \sup _{B(0,r)} \varphi). 
\end{align*}
So,
\begin{align}\label{eq:boundi1}
I_1 = \sup _{B(0,r)} \varphi - \varphi_{\partial B(0,r)} &\leq \left(\dfrac{3^{2n-1}}{ 2^{2n-2}}\right) \left(\sup _{B(0,r)} \varphi - \sup_{B(0,\frac{1}{2}r)}\varphi\right).
\end{align}

Now we estimate $I_2$. Set 
\begin{eqnarray} \label{for:bycen}
g(t):= \dashint _{ \partial B(0,e^t r)} \varphi d \sigma.  
\end{eqnarray}
We have 
\begin{eqnarray*}
\varphi _{B(0,r) }  &  =  & \dashint_{ B(0,r)}    \varphi  \\
&=& \frac{1 }{ |B(0,r)|} \int^0 _{-\infty}  \int _{\partial B(0,e^tr)} \varphi d \sigma d (e^t r )  \\
& =&  \frac{1 }{ |B(0,r)|}  \int_{-\infty}^0   |\partial B(0,r)|  g(t) d (e^t r ),   \ \  \  \text{ by}  \ \  ( \ref{for:bycen} ) \\
& =&  \int ^0 _{-\infty}   g(t) d (e^{2nt}).
\end{eqnarray*}
Since $g$ is convex and $d(e^{2nt}) $ is a probability measure on $(-\infty,0)$ with barycenter at $-\frac{1}{2n}$, it follows from Lemma \ref{barycenter} that 
$$
\int ^0 _{-\infty}   g(t) d (e^{2nt}) \geq g \left (- \frac{1} {2n} \right).
$$
Thus, we have 
\begin{eqnarray}  \label{eq:estimatei2}
I_2 & =  & \varphi_{ \partial B(0,r)}  -  \varphi_{B(0,r)}\\  \nonumber
& =& g(0) - \int_ {-\infty} ^0 g(t) d(e^{2nt} )\\  \nonumber
& \leq& g(0) - g\left(-\frac{1}{2n} \right)\\  \nonumber
& =&   \varphi_ { \partial B(0 , r)}  -   \varphi_ { \partial B\left(0,e^{-\frac{1} { 2n } }r \right)}. \nonumber
\end{eqnarray}
Combining  \eqref{eq:uomo}, \eqref{UO1}, \eqref{eq:boundi1}, and \eqref{eq:estimatei2}, we get  (\ref{eq:mo}).
\end{proof}

\begin{proof}[Proof of Theorem \ref{tmthm2}]

($ \Longrightarrow$ )  Let $U$ be a neighborhood of $\overline{\Omega}$ and  $\varphi\in\rm{PSH}(U)$. Rescaling $U$ if necessary, we may assume that there exists $V\Subset   U$ such that $\text{dist}(\overline{\Omega} , \partial V)>1$, 
so that for every $z\in \overline{\Omega}$, $B(z, 1) \subset \overline{V}$. We may also assume $\max  \limits_{\overline{V}} \varphi <-1$.

Since $\nu_\varphi(\overline{\Omega}) =0$, by definition and Lemma \ref {lelongnumber}, for every $a \in \overline{\Omega}$ and $\varepsilon>0$, there exists $0< r_a< \frac{1}{2}$
 \begin{eqnarray} \label{eq: lelong}
 \frac{ \sup \limits_ {B(a, \frac{r_a}{2}) }    \varphi} {\log \frac{r_a}{2}} < \varepsilon  \ \ \text{and}  \ \ \   \frac{ \sup \limits_ {\partial B(a, e^{-\frac{1}{2n} } r_a) }    \varphi} {\log (e^{-\frac{1}{2n} } r_a)} < \varepsilon.
 \end{eqnarray}

 The idea of the proof is to bound $\mathrm{MO}_{B(z,r) }(\varphi)$ uniformly from above with respect to $r$, for any $z\in{\Omega}$ such that $B(z,r)\Subset \Omega$. We then decompose $\mathrm{MO}_{B(z,r) }(\varphi)$ in terms of $\sup_{B(z,r)} \varphi$ by first fixing $z\in B(a,r_a)$ and taking a covering of $\overline{\Omega}$ by those balls.

First, note that for every $z\in \overline{\Omega}$, we have
 \begin{eqnarray*}\label{ineq:bijj}
  \frac{\sup \limits_ {B(z,r) } \varphi } {\log r} =  \frac{  \sup \limits_ {B(z,r) } \varphi    - \sup \limits_ {B(z,1) } \varphi   }{ \log r- \log 1 } + \frac{\sup \limits_ {B(z,1) } \varphi}{ \log r}.
  \end{eqnarray*}
  
Since $r  \rightarrow  \sup \limits _{B(0,r)} \varphi $ is convex and nondecreasing with respect to $\log r$, the slopes of $\log r  \mapsto \sup \limits_ {B(z,r) } \varphi $ are therefore nondecreasing and positive. Hence, the first right-hand term is also nondecreasing with respect to $r$, and since $\sup \limits_ {B(z,1) } \varphi $ is bounded, so is the second right hand term.  We then deduce that $ \frac{\sup \limits_ {B(z,r) } \varphi } {\log r}$ and   
$\sup \limits_ {B(z,r)} \varphi    - \sup \limits_ {B(z,\frac{r}{2}) } \varphi $
 are nondecreasing with respect to $r$. Similarly, we see that  $\frac{ \sup \limits_ {\partial B(z, e^{-\frac{1}{2n} } r) }    \varphi} {\log (e^{-\frac{1}{2n} } r)} $ and 
$ \displaystyle{ \sup \limits_ {\partial B(z,r ) } \varphi  - \sup \limits_ {\partial B(z, e^{-\frac{1}{2n} } r) }\varphi}$  are  nondecreasing with respect to $r$. Hence, for every $0< r\leq r_a$,
\begin{eqnarray*} \label{ineq:non-decra}
  \frac{\sup \limits_ {B(z,\frac{r}{2} ) } \varphi } {\log r} =   \frac{\sup \limits_ {B(z,\frac{r}{2}) } \varphi } {\log \frac{r}{2}}\cdot \frac{\log \frac{r}{2}}{ \log r } \leq   \frac{\sup \limits_ {B(z,r) } \varphi } {\log r} \cdot \frac{\log \frac{r}{2}}{ \log r } ,
\end{eqnarray*}
we get 
\begin{eqnarray}\label{eq: danzeng}
 \sup \limits_ {B(z,r ) } \varphi  \leq \sup \limits_ {B(z,\frac{r}{2} ) } \varphi  \cdot  \frac{ \log r }{\log \frac{r}{2}}. \end{eqnarray}
 Now, let $z\in B(a, \frac{r_a}{2})\cap\,\overline{\Omega}$ with $r_a  $ satisfying \eqref{eq: lelong}.   We have $B(a, \frac{r_a}{2}) \subset B(z, r_a)$, so
 \begin{eqnarray*}
 \sup \limits_ {B(z,r_a) } \varphi  \geq \sup \limits_ {B(a,\frac{r_a}{2}) } \varphi. 
\end{eqnarray*}
 We deduce the following. 
\begin{align} \label{eq:daxiao}
  \frac{\sup \limits_ {B(z,r_a) } \varphi } {\log r_a} &\leq   \frac{\sup \limits_ {B(a ,\frac{r_a}{2}) } \varphi } {\log r_a}  =  \frac{\sup \limits_ {B(a ,\frac{r_a}{2}) } \varphi } {\log \frac{r_a}{2} }  \cdot   \frac{\log \frac{r_a}{2} } {\log r_a}\nonumber\\
  &<\varepsilon \cdot \frac{\log \frac{r_a}{2}  }{\log r_a}   = \left(1+ \frac{ \log 2 }{ - \log r_a}  \right) \varepsilon \nonumber\\
  &\leq 2 \varepsilon 
\end{align}
by (\ref{eq: lelong}).
For every $z\in B(a, \frac{r_a}{2})\cap\overline{\Omega}$ and $r \leq   r_a  $, we have
 \begin{eqnarray*}
 \sup \limits_ {B(z,r ) } \varphi  - \sup \limits_ {B(z,\frac{r}{2}) } \varphi & \leq & \sup \limits_ {B(z,r_a ) } \varphi  - \sup \limits_ {B(z,\frac{r_a}{2}) } \varphi     
   \\ \nonumber 
& \leq & \sup \limits_ {B(z,\frac{r_a}{2}) } \varphi  \cdot  \left(   \frac{\log r_a}  {\log \frac{r_a}{2 } } -1\right) , \ \  \ \text{ (by \ \ \ref{eq: danzeng})}   \\ \nonumber 
&=&  \log 2  \cdot  \frac{\sup \limits_ {B(z,\frac{r_a}{2}) } \varphi } {\log \frac{r_a}{2} }  \\ \nonumber 
& <& (2 \log 2)    \varepsilon,   \ \  \ \text{ (by \ \ \ref{eq:daxiao})} . 
 \end{eqnarray*}
 Using the same idea, we have 
  \begin{eqnarray*}
 \sup \limits_ {\partial B(z,r ) } \varphi  - \sup \limits_ {\partial B(z, e^{-\frac{1}{2n} } r) } \varphi & \leq & \sup \limits_ {\partial B(z,r_a )} \varphi  - \sup \limits_ {\partial B(z, e^{-\frac{1}{2n} } r_a) } \varphi   \\ \nonumber 
& \leq & \sup \limits_ {\partial B(z, e^{-\frac{1}{2n} } r_a) } \varphi  \cdot  \left(   \frac{\log r_a}  {\log (e^{-\frac{1}{2n} } r_a) }-1\right)   \\ \nonumber  
&<& \frac{1}{2n}  \varepsilon.
 \end{eqnarray*}
 Hence, for every $z\in B(a, \frac{r_a}{2})\cap\,\overline{\Omega}$ and $r \leq \min \{  r_a,  \text{dist} (z, \partial \Omega) \}$,  we get from (\ref{eq:mo}) 
 \begin{eqnarray*} \label{eq:bijiao}
\mathrm{MO}_{B(z,r)}(\varphi) & \leq & 2  \left(\dfrac{3^{2n-1}}{ 2^{2n-2}}\right) \left(\sup _{B(z,r)} \varphi - \sup_{B(z,\frac{1}{2}r)}\varphi\right)  \\ \nonumber   
&& + 2 \left(  \varphi_ { \partial B(z , r)}  -   \varphi_ { \partial B\left(z,e^{-\frac{1} { 2n } }r \right)} \right)\\ \nonumber   
&<&   2  \left(\dfrac{3^{2n-1}}{ 2^{2n-2}}\right)     \cdot (2 \log 2) \varepsilon +  \frac{1}{n}  \varepsilon   \\ \nonumber   
& <& (  2^{2n+3}+1  ) \varepsilon.
\end{eqnarray*}

 We choose a finite covering of $\overline{\Omega}$ by balls  $B(a_i, \frac{r_{a_i}}{ 2} ) $, $1\leq i\leq m$, where $a_i\in \overline{\Omega}$. Denote $r_0 :=\min  \limits _{1\leq  i \leq m} \{ r_{a_i}  \}  $. Hence, for any $z\in \Omega$ and
$r < \min \{r_0 ,   \frac{1}{2}\text{dist} (z, \partial \Omega)    \}$, 
 
 \begin{eqnarray*} \label{eq:bijiao1}
\mathrm{MO}_{B(z,r)}(\varphi) & <&  ( 2^{2n+3}+1  ) \varepsilon.
\end{eqnarray*}
 We deduce that $\varphi\in \rm{VMO}(\Omega)$.\\

$(\Longleftarrow)$ Let $\varphi \in \mathrm{VMO}(\Omega)$.  We need to prove that $\nu_\varphi(\overline{\Omega}) =0$, that is, for every $a\in \overline{\Omega}$,   $\nu_\varphi(a)=0$. Due to Proposition \ref{co:vmoio}, it is sufficient to prove that  $\nu_\varphi(a)=0$ for every $a\in \partial \Omega$.

Let $a\in \partial\Omega$.  Take a sequence $\{z_j \} \subset \Omega$ such that $z_j \rightarrow a$.  
 Since $\varphi\in \mathrm{VMO}(\Omega)$,   by Lemma \ref{LM:BMOEPSILON},  for every $\varepsilon>0$,  there exists $r_0>0$ such that if $j \gg 1$, $r_j = \mathrm{dist} (z_j, \partial\Omega) < r_0$,  
\begin{eqnarray}\label{eq:vmoepsi2}
\| \varphi\|_{\mathrm{BMO}(B(z_j,r_j))  } <\varepsilon.
\end{eqnarray}
On the other hand,  for every $j$ fixed, applying John-Nirenberg inequality (\ref{JNE}) to $B(z_j, r_j)$, we have for every $B\subseteq B(z_j, r_j)$,
\begin{eqnarray} \label{ineq:uji2}
  \dashint_{B }  e^{c_n\frac{\left\vert \varphi-\varphi_{B } \right\vert}{2\Vert\varphi\Vert_{\mathrm{BMO}(B(z_j, r_j))}}}<  \infty.
\end{eqnarray}
Since $\left\vert \varphi-\varphi_{B(z_j, r_j) }\right\vert \geq |\varphi | - | \varphi_{B(z_j, r_j) } | = -\varphi + \varphi_{B(z_j, r_j) } $ for $B=B(z_j,r_j)$, we have from (\ref{eq:vmoepsi2}) and (\ref{ineq:uji2}),
 \begin{eqnarray}\label{eq:gujivaresi}
 \int_{B(z_j, r_j) }e^{-c_n\frac{  \varphi }   {2 \varepsilon }}   & \leq &  \int_{B(z_j, r_j) }  e^{ -\frac{c_n} {2\varepsilon} \varphi_{B(z_j, r_j) }}e^{c_n\frac{\left\vert \varphi-\varphi_{B(z_j, r_j) } \right\vert}{2 \varepsilon }}<  \infty. 
\end{eqnarray}

Since $\Omega$ satisfies the interior sphere condition and $z_j\to a$, we may assume that $a\in \partial B(z_j, r_j)$ for some $j\gg1$. Denote $\nu_\varphi(a) = \gamma$.  
By definition of the Lelong number of $\varphi$ at $a$,  there exist $r_a>0$ and a constant $C$, so that for every $z\in B(a, r_a)$,
$$
\varphi(z) \leq \gamma \log |z-a|  + C.
$$
Then (\ref{eq:gujivaresi}) implies that
 \begin{eqnarray*}
 \int_{B(z_j, r_j) }  \frac{ 1}   {|z-a| ^{ \frac{c_n} {2\varepsilon} \gamma }}  & \leq & e^ { \frac{-C\cdot c_n } {2 \varepsilon} }   \int_{B(z_j, r_j) }e^{-c_n\frac{  \varphi }   {2 \varepsilon }}   <\infty,  
\end{eqnarray*}
So, 
$ \frac{c_n} {2\varepsilon}   \gamma <2n $ i.e., 
$\nu_\varphi(a) \leq  \frac{4n}{c_n} \varepsilon $ for every $\varepsilon>0$.   
\end{proof}

\begin{remark}
The interior sphere condition is only required for the backward direction. As defined, the radius of the ball centered in $z_j$ might tend to $0$ as we are close to $a$. This geometric assumption is enough to avoid cusps pointing outwards the domain; however, they might point inward.
\end{remark}

Without any geometric assumption on the boundary of $\Omega$, we may deduce a weaker equivalence, which is Theorem \ref{tmthm1}.

\begin{proof} [Proof of Theorem \ref{tmthm1}]
$(\Longrightarrow)$  Assume that $\nu_\varphi(a)=0$ for every  $a\in\Omega$.
We only need to prove that for every $\Omega'\Subset \Omega$,  $\varphi \in \mathrm{VMO}(\Omega')$.

Let $\Omega'\Subset \Omega$ be any relatively compact subset. By assumption, $\nu_\varphi(\overline{\Omega'}) = 0$. 
Then, by Theorem \ref{tmthm2}, $\varphi \in \mathrm{VMO}(\Omega')$.

$(\Longleftarrow)$ By assumption, $\varphi\in\rm{VMO}(B)$ for any ball $B\Subset \Omega$. Let $a\in \Omega$, we can find a ball $B$ centered in $a$, contained in $\Omega$. By Proposition \ref{co:vmoio},
we conclude that $\nu_\varphi(a)=0$.
\end{proof}

\begin{remark}
    Since the Lelong number is defined locally, Theorem \ref{tmthm1} holds for any, not necessarily bounded, domain $\Omega\subset \mathbb{C}^n$. 
\end{remark}

We give below an example of a domain with H\"older boundary on which Theorem \ref{tmthm2} does not hold: there exists a (pluri)subharmonic VMO function $\varphi$ on $\Omega$ such that $\nu_\varphi(\overline{\Omega})=1$. This is because the interior sphere condition is not automatically guaranteed for H\"older domains.

\begin{example}\label{example1}
Let $\Omega : = \{x+iy \in\C\mid 0<x<1,  \vert y\vert <(1-x)^\alpha  \}$, $\alpha>1$. 
Then $\log |z-1| \in\mathrm{VMO}(\Omega)$, but $\nu_{\log|z-1|}(1)  = 1$.
\end{example}
Since $\log|z-1| $ is harmonic on $\Omega$, for every $a\in \Omega$ and for every ball $B(a,r) \Subset \Omega$ we have 
$$
\dashint_{B(a,r)} \log |z-1|= \log |a-1|.
$$ 
Now we let 
$$\Omega_1: = B(a,r) \cap B(1, |1-a|), \ \Omega_2: = B(a,r) \setminus  \bar{B}(1, |1-a|).$$ 
\begin{eqnarray*}
\int_{B(a,r)} \left\vert \log|z-1| - \log |a-1| \right\vert &=& \int_{\Omega_1} ( \log |a-1| - \log |z-1|)  \\ \nonumber   
&& + \int_{\Omega_2} ( \log |z-1| - \log |a-1|) \\ \nonumber   
  & =  &    (|\Omega_1| - |\Omega_2|) \log |a-1|  +  \int_{\Omega_2}  \log |z-1|    \\ \nonumber   
&&- \int_{\Omega_1}  \log |z-1| .
\end{eqnarray*}
On the other hand, for every  $z\in B(a,r)$,  $|a-1| -r \leq |z-1| \leq |a-1|+ r$,  so we have
\begin{eqnarray*}
  \log |a-1| + \log \left(1-\frac{r}{|a-1|}\right) \leq \log |z-1| \leq \log |a-1| + \log \left(1+ \frac{r} { |a-1|} \right) .
\end{eqnarray*}
Hence,
\begin{eqnarray*}
\int_{B(a,r)} \left\vert \log|z-1| - \log |a-1| \right\vert  & \leq   &    (|\Omega_1| - |\Omega_2|) \log |a-1|  +  |\Omega_2| \log |a-1|    \\ \nonumber   
  &  & +  |\Omega_2|  \log \left(1+ \frac{r} { |a-1|} \right)    -  |\Omega_1|   \log |a-1|    \\ \nonumber   
&&- |\Omega_1 |  \log (1- \frac{r} { |a-1|} ) \\ \nonumber   
  & =&  |\Omega_2|  \log \left (1+ \frac{r} { |a-1|} \right)   - |\Omega_1 |  \log \left(1- \frac{r} { |a-1|} \right). \\ \nonumber   
\end{eqnarray*}
So, 
\begin{eqnarray*}
  \dashint_{B(a,r)}  \left\vert \log|z-1| - \log |a-1|  \right\vert & \leq   &  \frac{  |\Omega_2| }{  |\Omega_1 |  +    |\Omega_2| } \log \left(1+ \frac{r} { |a-1|} \right)   \\ \nonumber   
&& - \frac{ |\Omega_1 | }{  |\Omega_1 |  +    |\Omega_2| }  \log \left (1- \frac{r} { |a-1|} \right)  \\ \nonumber   
  & \leq   &  \log \left (1+ \frac{r} { |a-1|} \right) -  \log \left(1- \frac{r} { |a-1|} \right)\\ \nonumber   
  & =   &  \log\frac{  |a-1| + r} {  |a-1|-r }\\ \nonumber   
  & \leq   &  \log\frac{  |a-1| +  |a-1| ^\alpha } {  |a-1|- |a-1| ^\alpha }  \\ \nonumber 
  &  =  &  \log\left (  \frac{1+ |a-1| ^{\alpha-1}  }{ 1- |a-1| ^{\alpha-1}} \right)
  \\ \nonumber 
  &\rightarrow &  0 \ \ \  ( a \rightarrow 1),
\end{eqnarray*}
for $ r\leq \text{dist}(a, \partial\Omega) \leq |a-1|^\alpha $.
Thus, for a given $\varepsilon >0$, there exists $\delta >0$, such that  for every $a\in \Omega$ satisfying  $|a-1|\leq  \delta$, and for $0< r\leq \delta^\alpha $,
\begin{eqnarray*}
\mathrm{MO}_{B(a,r)}( \log |z-1| ) < \varepsilon.  
\end{eqnarray*}
Denote  by $A = \{ a \in \Omega\mid  |a-1| \leq \delta \} $ and $B = \{ a \in \Omega\mid  |a-1| >   \delta  \}$.
Since $\log |z-1| $ is uniformly continuous on $\overline{B}$, $\log |z-1|\in\rm{VMO}(B)$.  That is, there exists $r_0>0$ such that for every $a\in B$ and $r < r_0$, we have 
\begin{eqnarray*}
\mathrm{MO}_{B(a,r)}( \log |z-1| ) < \varepsilon.  
\end{eqnarray*}
Hence, for every $a\in A\cup B$ and $r < \min\{  r_0,      \delta^\alpha  \}$,
\begin{eqnarray*}
\mathrm{MO}_{B(a,r)}( \log |z-1| ) < \varepsilon.  
\end{eqnarray*}
We conclude that $\log |z-1|\in\rm{VMO}(\Omega)$.

Combining the following lemma with Theorem \ref{tmthm2}, we can construct local  VMO  functions in any domain $\Omega\subset
\mathbb{C}^n$ from a plurisubharmonic function with zero Lelong number.

\begin{lemma}\label{compLelong}
Let $\varphi\in \rm{PSH} (\Omega)$ and $\chi$ be a convex increasing function in some interval containing $\varphi (\Omega)\cup\lbrace{-\infty}\rbrace$.  If $\chi'(-\infty) =0$ then $\nu_{\chi \circ \varphi} \equiv0$ on $\Omega$. 
\end{lemma}

\begin{proof}
We distinguish two cases depending on the nature of the point $a\in\Omega$: 

\textit{Case 1:} Assume that $\varphi$ is bounded from below near $a$.\\
Since $\chi$ is increasing, $\chi \circ \varphi  (a)$ is also bounded from below. Because $\log|z-a| \rightarrow - \infty$ as $z\to a$, we get 
 $$
  \frac{\chi \circ \varphi(z) }{ \log |z-a|}  \rightarrow 0 \ \ \  \text{as} \ \ \ \ z\rightarrow a.
 $$
Hence,  $\nu _{\chi \circ \varphi } (a) =0$.
 
\textit{Case 2:} Assume that $\lim\inf_{z\to a}\varphi(z)= - \infty$.\\
By the definition of the Lelong number, 
$$
\nu _{  \varphi} (a)    =   \liminf_{ z \rightarrow a} \frac{  \varphi(z) }{ \log |z-a|} \geq 0.
$$
Because $\chi$ is convex, the subgradient inequality applied to $\chi$ at $\varphi(z)$ implies,  for $z$ close enough to $a$,
$$
\chi(\varphi(z)) \geq \chi  (s) + c_s (\varphi(z) -s),
$$
where $s$ is a fixed point in the domain of $\chi$,  and $c_s$
 is the subgradient of $\chi$ at $s$.  Since $\chi'(-\infty) =0$,  for any $\varepsilon>0$ and for $z$ close enough to $a$, we can choose $s$ such that $c_s < \varepsilon$. 
 Thus, for $z$ close to $a$,
 $$
 \chi(\varphi(z)) \geq \chi  (s) + \varepsilon (\varphi(z) -s).
 $$
 Dividing by $\log |z-a|$, we get
 $$
  \frac{\chi \circ \varphi(z) }{ \log |z-a|} \leq   \frac{\chi (s) -\varepsilon s }{ \log |z-a|}  + \varepsilon  \frac{  \varphi(z) }{ \log |z-a|}.
 $$
 Taking liminf as $z\rightarrow a$,
 $$
 \nu _{\chi \circ \varphi } (a)  \leq  \varepsilon  \nu _{ \varphi } (a) .
 $$
Since $ \varepsilon >0$  is arbitrary, 
we deduce
 $$
 \nu _{\chi \circ \varphi } (a)  \leq 0 .
 $$
But since $\chi$ is convex and increasing,  $\chi \circ \varphi$ is plurisubharmonic and 

 $$
 \nu _{\chi \circ \varphi } (a)  \geq 0 .
 $$
 Thus,  $\nu _{\chi \circ \varphi } (a) =0$.

 In both cases,  we have  $\nu _{\chi \circ \varphi } (a) =0$ for any $a\in \Omega$.  Therefore,  $\nu _{\chi \circ \varphi } (a) \equiv 0$ on $\Omega$.
 
\end{proof}

By choosing various $\chi$ in Lemma \ref{compLelong}, we get examples of  locally VMO functions.\\

\begin{example}\label{example2}  Let $\varphi$ be a negative plurisubharmonic function in $\Omega\subset \C^n$.\\
\begin{enumerate}
\item  $- \dfrac{1}{\varphi}   \in \mathrm{VMO_{\rm{loc}}}(\Omega )$ (for $\chi(x):=-\frac 1 x$);
\item $- \log (-\varphi) \in \mathrm{VMO_{\rm{loc}}}(\Omega )$ (for $\chi(x):=-\log (-x)$);
\item  $- (-\varphi)^\alpha \in \mathrm{VMO_{\rm{loc}}}(\Omega )$,  $0< \alpha <1$ (for $\chi(x):= - (-x)^\alpha$, $ 0 < \alpha <1$). \\\medskip
\end{enumerate}

In particular, by taking $\varphi = \log \vert f\vert$ where $f$ is a holomorphic function on $\Omega$ such that $|f|<\frac 1 e $,  we deduce  
 \begin{enumerate}
\item $- \frac{1} {\log|f|}   \in \mathrm{VMO_{\rm{loc}}}(\Omega )$;
\item $ - \log (- \log |f|)    \in \mathrm{VMO_{\rm{loc}}}(\Omega )$;
 \item $  - (-\log |f|)^\alpha \in \mathrm{VMO_{\rm{loc}}}(\Omega )$, $0 < \alpha<1$.
\end{enumerate} 
\end{example}

\section{Sobolev space $\rm{W_{loc}^{1,2}}(\Omega)$  and  $ \rm{VMO_{loc}}(\Omega)$}

 The generalization of Lelong numbers in terms of positive closed currents, developed in \cite{Dem93}, have woven connections between the theory of the complex Monge-Amp\`ere operators and  Lelong numbers. 
The definitions of  both objects are extended in a way that makes wedge products of positive closed currents well-defined.
In this section, we construct plurisubharmonic functions for which the wedge product of the associated positive current is well-defined. We also give an example of a $\mathrm{VMO_{\text{loc}}}$  function for which the above wedge product is not well-defined, showing an obstacle to relate the zero mass conjecture to  VMO functions.\\

Let $\mathrm{W ^{k,p}}(\Omega) $ ( resp.  $\mathrm{W_{\text{loc}}^{k,p}}(\Omega) $ ) be the Sobolev space of all functions on a domain $\Omega\subset \R^{2n}$ whose  weak  derivatives of order $\leq k$ are in $ L  ^p (\Omega)$ (  resp. $ L_{\text{loc}}^p  (\Omega)$ ), for some $k\in \mathbb{Z}^+,  p\geq 1$.\\
 Let $\mathrm{W_0^{k,p}}(\Omega)$ be the completion of $C^\infty_0(\Omega)$ with respect to the $\mathrm{W^{k,p}}(\Omega)$-norm.

If $\varphi$ is a smooth plurisubharmonic function defined on an open subset of $\C^n$, the positive $(n,n)$-current $(dd^c \varphi)^n$ defines a nonnegative Radon measure, expressed in local coordinates as follows
\begin{align*}
(dd^c \varphi)^n & = dd^c \varphi \wedge \cdots \wedge  dd^c \varphi \\
&= c \det \left( \frac{\partial^2 \varphi}{ \partial z_j \partial \overline{z}_k} \right)idz_1\wedge d\bar{z}_1\wedge\dots\wedge idz_n\wedge d\bar{z}_n,
\end{align*}
where $d = \partial + \overline{\partial}$,  $d^c = \frac{i}{\pi} (  \overline{\partial} -\partial )$ and $c>0$ is a normalizing constant.
The \textit{complex Monge-Amp\`ere operator} refers to the non-linear partial differential operator $\varphi\mapsto \det \left( \frac{\partial^2 \varphi}{ \partial z_j \partial \bar{z}_k} \right)$.  \\

An example given by Shiffman and Taylor (see Appendix 1 in \cite{Siu1975}) shows that the complex  Monge-Amp\`ere operator cannot be well-defined as a non-negative Radon measure for an arbitrary plurisubharmonic function. However,   $(dd^c \varphi)^n$  can be well-defined if $\varphi$ is plurisubharmonic
and locally bounded (see \cite{BT1982}). The local boundedness can be relaxed by looking at the dimension of the \textit{unbounded locus} of $\varphi$, that is, the set of points $x$ where $\varphi$ is unbounded in a neighborhood of $x$. Demailly \cite{Dem93} proves that it is enough to assume that the unbounded locus is relatively compact in the domain of the definition of $\varphi$. \\

Let $\Omega\subset \mathbb{C}^n$ be an open set. Denote by $\mathcal{D} (\Omega)$ the set of functions in $\mathrm{PSH}(\Omega)$ for which the  complex  Monge-Amp\`ere operator is well-defined. 

 If $n=1$, the complex Monge-Amp\`ere operator is the linear Laplace operator. If $\varphi$ is a subharmonic function that is not identically equal to $-\infty$ on any connected component of $\Omega$, then $dd^c\varphi$ defines a positive Borel measure and belongs to  $\mathcal{D}(\Omega)$ (see, for example, \cites{DGZ2016,GZ2007}).

If $n=2$,  we have the following characterization.
\begin{theorem}[B\l ocki, \cite{B2004}]
If $\Omega$ is an open subset of $\C^2$, then $\mathcal{D} (\Omega) = \rm{PSH}(\Omega)\cap \mathrm{W_{\rm{loc}}^{1,2}}(\Omega)$.
\end{theorem}

It is natural to ask for the relation between   $\mathrm{VMO_{\text{loc}}}(\Omega)$ and $\mathcal{D}(\Omega)$ for a bounded domain $\Omega\subset \mathbb{C}^n$, since both families of functions involve $dd^c\varphi$ for a plurisubharmonic function $\varphi$. In one complex dimension, any function in $\mathcal{D}(\Omega)$ is in   $\mathrm{VMO_{\text{loc}}}(\Omega)$ (see for example Proposition 6 in \cite{Vig2007}).  
However, in higher dimensions, the functions in $\mathrm{W_{\text{loc}}^{1,2}}(\Omega)$ are not in general in $\mathrm{VMO_{\text{loc}}}(\Omega)$ \cite{Vig2007}. We prove that the converse does not hold either (see Remark \ref{rmk:notW}). \\

Below, we use Proposition \ref{th:tcompo1} and Proposition \ref{th:tcompo2} to construct functions in $\mathrm{W_{\text{loc}}^{1,2}}(\Omega)$ from a plurisubharmonic function $\varphi$ thanks to the idea of  Lemma \ref{compLelong}.

 Let $t(x) : = - \log (-x)$  
 defined on $(-\infty, 0)$.  
 For any integer $m$,  we denote $t^{(m)} :=\underbrace{ t \circ t \circ \cdots \circ t}_{m}$.
  \begin{proposition} \label{th:tcompo1}
 Let $\psi < - \gamma< 0$ be a subharmonic function on  a domain $\Omega \subset \R^{2n}$,  where $\gamma$ is sufficiently large.
Then $t^{(m)} ( \psi) \in \mathrm{W_{\rm{loc}}^{1,2}}(\Omega)$ for any integer $m$.
  \end{proposition}

    \begin{proposition} \label{th:tcompo2}
 Let $\psi < - \gamma< 0$ be a subharmonic function on  a domain $\Omega \subset  \R^{2n}$. 
Then $- ( -\psi)^\alpha \in \mathrm{W_{\rm{loc}}^{1,2}}(\Omega)$ for every $0\leq \alpha <\frac 1 2$.
  \end{proposition}
  \begin{remark}\label{rmk:notW}
  Note that the exponent $\frac 1 2$ in Proposition \ref{th:tcompo2} is sharp.\\\medskip
  Let $\mathbb B^2_r:= \{(z_1,z_2)\in\mathbb{C}^2\mid \vert z_1\vert ^2 + |z_2|^2 <r \}$ for $r \ll 1$
and
  $ \psi(z_1,z_2) : =   \log |z_1|^2 $. So, $\psi<-\gamma<0$ 
  and
\begin{eqnarray*}
  \nabla ( - (-\log |z_1|^2)^{\frac{1}{2}}) = \frac{1}{2} \frac{1}{z_1} (-\log |z_1|^2)^{-1/2}  \notin L^2_{\rm{loc}}(\mathbb B^2_r)
  \end{eqnarray*}
Thus,  $ - (-\log |z_1|^2)^{\frac{1}{2}}\notin \mathrm{W_{\rm{loc}}^{1,2}}(\mathbb B^2_r)$.  On the other hand, from Example \ref{example2}, $- (-\log|z_1|^2)^{\alpha} \in \mathrm{VMO_{\rm{loc}}}(\mathbb B^2_r)$ for any $0<\alpha <1$.
  \end{remark}

In order to prove Proposition \ref{th:tcompo1} and  Proposition  \ref{th:tcompo2},  we need the following lemmas.

\begin{lemma} [\cite{Chen2020},  Lemma 3.2]
Let $\psi$,  $\phi$ be $C^2$ functions on a domain $\Omega\subset \R^{2n}$.  Let $\eta : \R \rightarrow (0, +\infty)$ be a $C^1$ function such that $\eta' >0$. If either $\phi$ or $\psi$ has compact support in $\Omega$,  then 

\begin{eqnarray}  \label{eq:elmenteq}
\int _{\Omega}    \phi^2 \left [   \frac{2 \Delta \psi }{ \eta(-\psi) }   + \frac{ \eta'(-\psi) }{  \eta^2(-\psi) }    | \nabla \psi    | ^2  \right ]  \leq 4 \int_{\Omega } \frac{ | \nabla \phi    | ^2  }{ \eta'(-\psi)}  .
\end{eqnarray}
\end{lemma}

\begin{remark}
If we further assume that $\psi$ is a subharmonic function, then we get from \eqref{eq:elmenteq}
\begin{eqnarray}  \label{eq:elmenteq11}
\int _{\Omega}    \phi^2   \frac{ \eta'(-\psi) }{  \eta^2(-\psi) }    | \nabla \psi    | ^2   \leq 4 \int_{\Omega } \frac{ | \nabla \phi    | ^2  }{ \eta'(-\psi)}.
\end{eqnarray}

\end{remark}

\begin{lemma} [\cite{C2017}, Proposition 4.2] \label{le:mainle}
Let $\psi < - \gamma< 0$ be a subharmonic function on a domain $\Omega \subset \R^{2n}$.  Let $\eta : \R \rightarrow (0, +\infty)$ be a $C^1$ function such that $\eta'$ is a positive non-increasing function with $\frac{1}{\eta'(-\psi)} \in L^1_{\rm{loc}} (\Omega)$.  Then $\kappa_\eta(-\psi) \in  \mathrm{W_{\rm{loc}}^{1,2}}(\Omega)$ where $\kappa_\eta(t) : = \int_\gamma^t \frac{ \sqrt{\eta'(s)}}{ \eta(s)} ds$, $t \geq \gamma.$ Moreover, one has 
\begin{eqnarray} \label{eq:tidu1}
\int _{\Omega}   \phi^2\frac{ \eta'(-\psi) }{  \eta^2(-\psi) }    | \nabla \psi    | ^2   \leq 4 \int_{\Omega } \frac{ | \nabla \phi    | ^2  }{ \eta'(-\psi)} , \ \ \  \phi \in \mathrm{W_0^{1,2}}(\Omega).
\end{eqnarray}

\end{lemma}

\begin{proof} [Proof of Proposition \ref{th:tcompo1}]
Since $\psi$ is subharmonic, as a composition with a convex function, $-\log(-\psi)$ is subharmonic. Let $\eta (t): = t$ and $\kappa_{\eta} (t):  = \log t$. By Lemma \ref{le:mainle}, $-\log(-\psi) \in \mathrm{W_{\text{loc}}^{1,2}}(\Omega)$.  If $\gamma$ is sufficiently large so that $t^{(m)} (\psi)$ is still well-defined, we repeat $m$ times Lemma \ref{le:mainle} to get the result.

\end{proof}

\begin{proof} [Proof of Proposition \ref{th:tcompo2}]
Let $\eta(t) :  = t^{1-2\alpha}$,  $0<\alpha <\frac 12$.  Then
$$
\frac{1}{ \eta'(-\psi)} = \frac{1}{1-2 \alpha}(-\psi)^{2\alpha} \in L_{\text{loc}}^1(\Omega),
$$
and
\begin{eqnarray*}
\kappa_{\eta}(t) &=&  \int_0^t \frac{ \sqrt{\eta'(s)}}{ \eta(s)} ds\\ 
&=&\int_0^t \frac{\sqrt{(1-2\alpha) s^{-2\alpha}}}{s^{1-2\alpha}}ds\\
&=& \sqrt{1-2\alpha} \int_0^t  \frac{1}{ s^{1-\alpha} }ds\\
&=& \frac{\sqrt{1-2\alpha} }{\alpha} t^\alpha.
\end{eqnarray*}
By Lemma \ref{le:mainle}, $ (-\psi)^\alpha \in \mathrm{W_{\text{loc}}^{1,2}}(\Omega)$ for every $0<\alpha<\frac 1 2$.
For $\alpha=0$, one trivially has $ (-\psi)^\alpha \in \mathrm{W_{\text{loc}}^{1,2}}(\Omega)$. 
\end{proof}

\section*{Acknowledgements}
We thank Prof.  Bo-Yong Chen,  Dr.  Xu Wang and Dr.  YuanPu Xiong for valuable discussions. We also thank the referees for pointing out inaccuracies in the previous version and for their valuable remarks.


\end{document}